\documentclass[leqno,12pt]{article}
\usepackage{amsmath, amssymb,soul}
\usepackage{amsthm} 

\theoremstyle{plain} 
\newtheorem{theorem}{\indent\sc Theorem}[section]
\newtheorem{lemma}[theorem]{\indent\sc Lemma}

\newtheorem{proposition}[theorem]{\indent\sc Proposition}

\numberwithin{equation}{section}
\theoremstyle{definition} 

\newtheorem{example}[theorem]{\indent\sc Example}

\newcommand\on{\operatorname}
\renewcommand\div{\on{div}}

\newcommand\Hess{\on{Hess}}
\newcommand\Ric{\on{Ric}}

\newcommand\trace{\on{trace}}


\def\({\left( }
\def\){\right)}
\def\<{\left< }
\def\>{\right>}

\title{On conformal collineation and almost Ricci solitons}
\author{Adara M. Blaga and Bang-Yen Chen}
\date{}

\begin{document}

\maketitle

\markboth{{\small\it {\hspace{1cm} On conformal collineation and almost Ricci solitons}}}
{\small\it{On conformal collineation and almost Ricci solitons \hspace{1cm}}}


\begin{abstract}
We provide conditions for a Riemannian manifold with a nontrivial closed affine conformal Killing vector field to be isometric to a Euclidean sphere or to the Euclidean space. Also, we formulate some triviality results for almost Ricci solitons with affine conformal Killing potential vector field.
\end{abstract}

\section{Preliminaries}

Stationary solutions to Ricci flow \cite{ha}, Ricci solitons have recently been more and more intensively studied in (semi-)Rie\-man\-ni\-an geometry.
A (semi-)Rie\-man\-ni\-an manifold $(M,g)$ is said to be
\textit{a Ricci} (respectively, \textit{an almost Ricci}) \textit{soliton} \cite{ha,pi} if there exist a smooth vector field $\xi$ and a real constant (respectively, a smooth function) $\lambda$ such that
\begin{equation}\label{1}
\frac{1}{2}\mathcal L_{\xi}g+\Ric=\lambda g,
\end{equation}
where $\mathcal L_{\xi}$ is the Lie derivative operator in the direction of $\xi$ and $\Ric$ is the Ricci curvature tensor.
Also, a soliton is called \textit{trivial} if the potential vector field $\xi$ is a Killing vector field, that is, $\mathcal L_{\xi}g=0$. We remark that the trivial Ricci solitons are just the Einstein manifolds, that is, they satisfy $\Ric=\frac{r}{n}g$, where $r$ denotes the scalar curvature.

Almost Ricci soliton equation is closely related to the affine conformal Killing vector fields'.
Precisely, if $(M,g)$ is a Riemannian manifold and $\nabla$ is the Levi-Civita connection of $g$, then,
a smooth vector field $\xi$ on $M$ is said to be \cite{ya}

(i) a \textit{Killing vector field} if
$\mathcal L_{\xi}g=0;$

(ii) a \textit{conformal Killing vector field} if
$\mathcal L_{\xi}g=2\rho g$ for $\rho$ a smooth function on $M$;

(iii) an \textit{affine Killing vector field} if
$\mathcal L_{\xi}\nabla=0;$

(iv) an \textit{affine conformal Killing vector field} if
$\mathcal L_{\xi}\nabla=d\rho\otimes I+I\otimes d\rho-g\otimes \nabla \rho$
for $\rho$ a smooth function on $M$, where $\nabla \rho$ is the gradient of $\rho$.

\bigskip

Yano \cite{yan} proved that any affine Killing vector field is Killing, too.

\bigskip

In \cite{kat}, Katzin, Levine and Davis characterized the affine conformal Killing vector fields
and proved that $\xi$ is an affine conformal Killing vector field on a Riemannian manifold $(M,g)$ with
\begin{equation}\label{h}
\mathcal L_{\xi}\nabla=d\rho\otimes I+I\otimes d\rho-g\otimes \nabla \rho,
\end{equation}
if and only if
\begin{equation}\label{2}
\mathcal L_{\xi}g=2\rho g+K, \ \ \nabla K=0
\end{equation}
with $K$ a symmetric $(0,2)$-tensor field. If there exists such a vector field satisfying (\ref{h}) (or, equivalently, (\ref{2})), then it is said that $M$
admits a symmetry, called \textit{conformal collineation}, generated by $\xi$, condition appearing in applications from relativity theory.
We remark that the first equation from relation (\ref{2}) is just the almost Ricci soliton equation for $K=-2\Ric$. Therefore, the potential vector field of an almost Ricci soliton with parallel Ricci tensor is an affine conformal Killing vector field.

\bigskip

Since spheres posses affine conformal Killing vector fields, it is appropriate to characterize the Euclidean spheres using them.
On the other hand, the well-known theorems of Obata and Tashiro provide useful characterizations for Euclidean spheres and for the Euclidean space,
precisely, denoting by $\Hess(f)$ the hessian $(0,2)$-tensor field of a smooth function $f$ on a Riemannian manifold $(M,g)$, defined by
$\Hess(f)(X,Y):=g(\nabla_X\nabla f,Y)$ for vector fields $X,Y\in {\mathfrak X}(M)$, we have

\begin{theorem}\label{obata} \cite{ob}
If there exists a non constant smooth function $f$ on a complete Riemannian manifold $(M,g)$ satisfying $\Hess(f)=-cfg$ with $c$ a positive constant, then $M$ is isometric to the Euclidean sphere of radius $\frac{1}{\sqrt{c}}$.
\end{theorem}

\begin{theorem}\label{tasi} \cite{tas}
If there exists a non constant smooth function $f$ on a complete Riemannian manifold $(M,g)$ satisfying $\Hess(f)=cg$ with $c$ a non zero constant, then $M$ is isometric to the Euclidean space.
\end{theorem}

It is known that the sphere with the canonical metric is an almost Ricci soliton. Two natural questions arise in this context:

{\it When an almost Ricci soliton is trivial?} and

{\it Under which conditions, a compact almost Ricci soliton is isometric to a sphere?}

Having these in mind, in the present paper, based on Obata's and Tashiro's theorems, we shall characterize Euclidean spheres and the Euclidean space by means of affine conformal Killing vector fields.
Then, we prove some triviality results for almost Ricci solitons with this kind of potential vector field. We remark that an almost Ricci soliton of dimension $>2$ with an affine Killing or conformal Killing potential vector field is an Einstein manifold, and we show that under certain assumptions, a connected almost Ricci soliton with affine conformal Killing potential vector field is an Einstein manifold, too, and the vector field is affine Killing.

\section{Characterizing the Euclidean spheres}

Let $\xi$ be a vector field satisfying (\ref{2}). If we define the $(1,1)$-tensor field $H$ by
\linebreak
$g(HX,Y):=K(X,Y)$ for $X,Y\in {\mathfrak X}(M)$, then we find from $\nabla K=0$ that $\trace(H)$ is a constant.
Denoting by $\eta$ the $g$-dual $1$-form of $\xi$ and considering the skew-symmetric $(1,1)$-tensor field $J$ defined by $g(JX,Y):=\frac{1}{2}(d\eta)(X,Y)$ for $X,Y\in {\mathfrak X}(M)$, relations (\ref{2}) can be written as
\begin{equation*}\label{3}
\nabla \xi=\rho I+J+\frac{1}{2}H,  \ \ \nabla H=0.
\end{equation*}
We remark that if $\xi$ is closed and $H=0$, then $\xi$ is precisely a concircular vector field.

Let $R$, $\Ric$, $Q$ and $r$ be the Riemannian curvature tensor, Ricci curvature tensor, Ricci operator and scalar curvature of an $n$-dimensional compact Rieman\-ni\-an manifold $(M,g)$. By a direct computation and taking into account that
$\nabla H=0$, we derive
\begin{equation*}
R(\cdot\, ,\cdot)\xi=d\rho\otimes I+I\otimes d\rho +d^J,
\end{equation*}
where $d^J(X,Y):=(\nabla_XJ)Y-(\nabla_YJ)X$ for $X,Y\in {\mathfrak X}(M)$, and, by tracing it, we infer
\begin{equation*}
Q\xi=-(n-1)\nabla \rho-\sum_{i=1}^n(\nabla_{E_i}J)E_i,
\end{equation*}
where $\{E_i\}_{1\leq i\leq n}$ is a local orthonormal frame on $(M,g)$.
Also, by taking the trace in (\ref{2}), we get
\begin{equation}\label{15}
\div(\xi)=n\rho+\frac{1}{2}\trace(H),
\end{equation}
which, by the divergence theorem, implies
\begin{equation}\label{7}
\int_M(2n\rho+\trace(H))=0.
\end{equation}
From
\begin{equation*}
\div(\rho\xi)=\rho\div(\xi)+g(\nabla \rho,\xi)
\end{equation*}
we find
\begin{equation*}
\int_Mg(\nabla \rho,\xi)=-\frac{1}{2}\int_M\rho(2n\rho+\trace(H)).
\end{equation*}
Taking into account that $\trace(HJ)=0$, we get
\begin{equation*}
\div(J\xi)=-|J|^2+(n-1)g(\nabla \rho,\xi)+\Ric(\xi,\xi),
\end{equation*}
which implies
\begin{equation}\label{10}
\int_M\Ric(\xi,\xi)=\int_M\Big(|J|^2+\frac{n-1}{2}\rho(2n\rho+\trace(H))\Big).
\end{equation}
Computing $\div(Q\xi)$, we infer
\begin{equation*}
\div(Q\xi)=\rho r+\frac{1}{2}\xi(r)+\frac{1}{2}\trace(HQ)
\end{equation*}
and we obtain
\begin{equation*}
\int_M\big(2\rho r+\xi(r)+\trace(HQ)\big)=0.
\end{equation*}
From
\[
\div(\rho Q\xi)=\rho\div(Q\xi)+g(\nabla \rho,Q\xi)
=\frac{\rho}{2}\big(2\rho r+\xi(r)+\trace(HQ)\big)+\Ric(\nabla \rho,\xi)
\]
we find
\begin{equation}\label{4}
\int_M \Ric(\nabla \rho,\xi)=-\frac{1}{2}\int_M\rho\big(2\rho r+\xi(r)+\trace(HQ)\big).
\end{equation}
Using
\[
\div(\Delta(\rho)\xi)=\Delta(\rho)\div(\xi)+g(\nabla (\Delta(\rho)),\xi)
\]
\[
=\frac{n}{2}\Delta (\rho^2)-n|\nabla \rho|^2+\frac{\trace(H)}{2}\Delta(\rho)+g(\nabla (\Delta(\rho)),\xi),
\]
and taking into account that $\trace(H)$ is constant, we get
\begin{equation}\label{5}
\int_M g(\nabla (\Delta(\rho)),\xi)=n\int_M|\nabla \rho|^2.
\end{equation}

Let $H_{\rho}$ be the symmetric $(1,1)$-tensor field defined by $g(H_{\rho}X,Y):=\Hess(\rho)(X,Y)$ for $X,Y\in {\mathfrak X}(M)$. Since
$\trace(JH_{\rho})=0$, we have
\begin{equation*}
g(\nabla (\Delta(\rho)),\xi)+\Ric(\nabla \rho,\xi)=\sum_{i=1}^ng((\nabla_{E_i}H_{\rho})E_i,\xi)
\end{equation*}
\[
=\sum_{i=1}^ng((\nabla_{E_i}H_{\rho})\xi,E_i):=\sum_{i=1}^ng(\nabla_{E_i}H_{\rho}\xi,E_i)-\sum_{i=1}^ng(\nabla_{E_i}\xi,H_{\rho}E_i)
\]
\[
=\div(H_{\rho}\xi)-\rho \sum_{i=1}^ng(E_i,H_{\rho}E_i)-\sum_{i=1}^ng(J{E_i},H_{\rho}E_i)-\frac{1}{2}\sum_{i=1}^ng(H{E_i},H_{\rho}E_i)
\]
\[
=\div(H_{\rho}\xi)-\rho\Delta(\rho)-\frac{1}{2}\trace(HH_{\rho})
\]
\[=\div(H_{\rho}\xi)-\frac{1}{2}\Delta(\rho^2)+|\nabla \rho|^2-\frac{1}{2}\trace(HH_{\rho})
\]
and we obtain
\begin{equation*}
\int_M |\nabla \rho|^2=\int_Mg(\nabla (\Delta(\rho)),\xi)+\int_M\Ric(\nabla \rho,\xi)+\frac{1}{2}\int_M\trace(HH_{\rho}),
\end{equation*}
which, by means of (\ref{4}) and (\ref{5}), implies
\begin{equation*}
2(n-1)\int_M |\nabla \rho|^2=\int_M\Big(\rho\big(2\rho r+\xi(r)+\trace(HQ)\big)-\trace(HH_{\rho})\Big).
\end{equation*}
Then we can state

\begin{proposition}
Let $\xi$ be a vector field on an $n$-dimensional compact Riemannian manifold $(M,g)$ satisfying (\ref{2}). Then
\begin{equation}\label{6}
2(n-1)\int_M |\nabla \rho|^2=\int_M\Big(\rho(2\rho r+\xi(r))+\trace(H\circ (\rho Q-H_{\rho}))\Big).
\end{equation}
\end{proposition}

\noindent And we conclude

\begin{proposition}
An affine conformal Killing vector field $\xi$ on an $n$-dimen\-sion\-al compact and connected Riemannian manifold $(M,g)$ satisfying (\ref{2}) and
\[
\int_M\trace(H\circ (H_{\rho}-\rho Q))\geq \int_M\rho(2\rho r+\xi(r))
\]
is an affine Killing vector field.
\end{proposition}

Denote by $\lambda_1$ be the first non zero eigenvalue of $-\Delta$. Then we can prove
\begin{theorem}
Let $\xi$ be a nontrivial closed (in particular, gradient) affine conformal Killing vector field on an $n$-dimensional compact Riemannian manifold $(M,g)$, $n>1$, of constant scalar curvature $r\neq 0$ satisfying (\ref{2}).
If $\lambda_1\geq nc$,
\[
\trace(H)=\trace(HQ)=\trace(HH_{\rho})=0
\]
and
\[
\Ric\big(\nabla\rho+c\xi,\nabla\rho+c\xi\big)\geq 0
\]
for $c=\frac{r}{n(n-1)}$, then $M$ is isometric to a Euclidean sphere.
\end{theorem}
\begin{proof}
Since $\xi$ is closed, $d\eta=0$ and $J=0$. From (\ref{6}) and (\ref{7}) we have
\begin{equation*}
(n-1)\int_M |\nabla \rho|^2=r\int_M\rho^2
\end{equation*}
(so $r> 0$) and
\begin{equation*}
\int_M \rho=0.
\end{equation*}
By the minimum principle and from the hypotheses, we get
\[
r\int_M\rho^2\leq (n-1)\lambda_1\int_M\rho^2\leq (n-1)\int_M|\nabla \rho|^2=r\int_M\rho^2,
\]
so, the equality holds if and only if $\Delta(\rho)=-\lambda_1 \rho$. Also, $r=(n-1)\lambda_1$ and we can write
$\Delta(\rho)=-nc\rho$.

Further, following the same steps as in the proof of Theorem 1.1 from \cite{sharief}, we get the conclusion by Obata's Theorem \ref{obata}.
\end{proof}

If we drop the closeness condition on $\xi$, we can state the previous theorem for a more general case. Precisely
\begin{proposition}
Let $\xi$ be a nontrivial affine conformal Killing vector field on an $n$-dimensional compact Riemannian manifold $(M,g)$, $n>1$, of constant scalar curvature $r\neq 0$ satisfying (\ref{2}).
If $\lambda_1\geq nc$,
\[
\trace(H)=0, \ \ \int_M\trace(HH_{\rho})=\int_M\rho\trace(HQ),
\]
\[
c\int_M|J|^2\leq \int_M\trace\Big(HH_{\rho}-\frac{(n-1)c\rho}{2}H\Big)
\]
and
\[
\Ric\big(\nabla\rho+c\xi,\nabla\rho+c\xi\big)\geq 0
\]
for $c=\frac{r}{n(n-1)}$, then $M$ is isometric to a Euclidean sphere.
\end{proposition}
\begin{proof}
By a direct computation, we obtain
\[
\Ric(\nabla\rho,\nabla\rho)=\div(H_{\rho}(\nabla \rho))-|H_{\rho}|^2-g(\nabla(\Delta(\rho)),\nabla \rho)
\]
and taking into account that $\Delta(\rho)=-nc\rho$, $\int_M|\nabla \rho|^2=nc\int_M\rho^2$, by means of
Schwartz's inequality, $|H_{\rho}|^2\geq \frac{(\Delta(\rho))^2}{n}$, we get
\[
\int_M\Ric(\nabla\rho,\nabla\rho)=-\int_M|H_{\rho}|^2+nc\int_M|\nabla \rho|^2\leq n(n-1)c^2\int_M\rho^2.
\]
Then, together with (\ref{4}) and (\ref{10}) imply
\[
\int_M\Ric\big(\nabla\rho+c\xi,\nabla\rho+c\xi\big)=\int_M\Big(\Ric(\nabla\rho,\nabla\rho)+2c\Ric(\nabla\rho,\xi)+c^2\Ric(\xi,\xi)\Big)
\]
\[
\leq
n(n-1)c^2\! \int_M\rho^2
-c\int_M\rho\big(2\rho r+\xi(r)+\trace(HQ)\big)
\]
\[
+c^2 \! \int_M\! \Big(|J|^2+\frac{n-1}{2}\rho(2n\rho+\trace(H))\Big)
\]
\[
=c^2\int_M\Big(|J|^2+\frac{(n-1)\rho}{2}\trace(H)-\frac{\rho}{c}\trace(HQ)\Big)
\]
and using the hypotheses, we get the conclusion.
\end{proof}

\begin{example}
In \cite{du}, Duggal showed that, for $K:=\omega\otimes \eta+\eta\otimes \omega$ in (\ref{2}) with $\omega$ a $1$-form and $\eta$ the $g$-dual $1$-form of $\xi$, the vector field $\xi$ is affine conformal Killing.
\end{example}

If we assume that $\xi$ satisfies (\ref{2}) and if we denote by $\zeta$ the $g$-dual vector field of $\omega$, then
\[
H=\omega\otimes \xi+\eta\otimes \zeta
\]
and
\[
\trace\big(H\circ(\rho Q-H_{\rho})\big)=2\big(\rho \Ric-\Hess(\rho)\big)(\xi,\zeta).
\]
Then we can state

\begin{proposition}
Let $\xi$ be a vector field on an $n$-dimensional compact Riemannian manifold $(M,g)$ satisfying (\ref{2}) with $K:=\omega\otimes \eta+\eta\otimes \omega$. Then
\begin{equation*}
2(n-1)\int_M |\nabla \rho|^2=\int_M\rho(2\rho r+\xi(r))+2\int_M\big(\rho \Ric-\Hess(\rho)\big)(\xi,\zeta),
\end{equation*}
where $\zeta$ the $g$-dual vector field of the $1$-form $\omega$.
\end{proposition}

Since
\[
\trace(H)=2g(\xi,\zeta), \ \trace(HQ)=2\Ric(\xi,\zeta), \ \trace(HH_{\rho})=2\Hess(\rho)(\xi,\zeta),
\]
if we denote by $\lambda_1$ the first non zero eigenvalue of $-\Delta$, we can state

\begin{proposition}
Let $\xi$ be a nontrivial closed (in particular, gradient) affine conformal Killing vector field on an $n$-dimensional compact Riemannian manifold $(M,g)$, $n>1$, of constant scalar curvature $r\neq 0$ satisfying (\ref{2}) with $K:=\omega\otimes \eta+\eta\otimes \omega$, and let $\zeta$ be the $g$-dual vector field of the $1$-form $\omega$. If
$\lambda_1\geq nc$,
\[
g(\xi,\zeta)=\Ric(\xi,\zeta)=\Hess(\rho)(\xi,\zeta)=0
\]
and
\[
\Ric\big(\nabla\rho+c\xi,\nabla\rho+c\xi\big)\geq 0
\]
for $c=\frac{r}{n(n-1)}$, then $M$ is isometric to a Euclidean sphere.
\end{proposition}

\section{Characterizing the Euclidean space}

Consider now $\xi$ a nontrivial closed vector field on an $n$-dimensional connected Riemannian manifold $(M,g)$ satisfying (\ref{2}), and assume that
$(M,\nabla f,\lambda,g)$ is a gradient almost Ricci soliton. Using the notations from the previous section, firstly we prove the next lemma.

\begin{lemma}
The following relations hold
\begin{equation}\label{c}
\xi(\rho)\xi=|\xi|^2\nabla \rho, \ \
\xi(\rho)\eta=|\xi|^2d\rho.
\end{equation}
\end{lemma}
\begin{proof}
We consequently have
\begin{equation*}
\nabla \xi=\rho I+\frac{1}{2}H, \ \
\nabla (|\xi|^2)=2\rho \xi+H\xi,
\end{equation*}
\begin{equation*}
\Hess(|\xi|^2)=2\Big(\rho^2g+d\rho\otimes \eta+\rho K+\frac{1}{4}K^2\Big),
\end{equation*}
where $K^2(X,Y):=g(HX,HY)$ for $X,Y\in {\mathfrak X}(M)$, from where we deduce that
$d\rho\otimes \eta$
is a symmetric $(0,2)$-tensor field, hence, $d\rho(\xi)\eta(X)=d\rho(X)\eta(\xi)$ for any $X\in {\mathfrak X}(M)$, and we get the conclusion.
\end{proof}

Also, we have
\begin{equation*}
Q\xi=-(n-1)\nabla \rho,
\end{equation*}
\begin{equation*}
\nabla(\xi(f))=\lambda \xi+(n-1)\nabla \rho+\rho\nabla f
+\frac{1}{2}H(\nabla f),
\end{equation*}
\begin{equation*}
\Hess(\xi(f))(Y,Z)=\Big(\lambda \rho g+\frac{\lambda}{2}K+(n-1)\Hess(\rho)+\rho\Hess(f)\Big)(Y,Z)
\end{equation*}
\begin{equation*}
+\frac{1}{2}\Hess(f)(Y,HZ)
+(d\lambda\otimes \eta+d\rho\otimes df)(Y,Z)
\end{equation*}
for any $Y,Z\in {\mathfrak X}(M)$, and since $\Hess(f)(HY,Z)=\Hess(f)(Y,HZ)$, we deduce that
\[
T:=|\xi|^2|\nabla f|^2(d\lambda\otimes \eta+d\rho\otimes df)
\]
is a symmetric $(0,2)$-tensor field, hence, $T(\xi,X)=T(X,\xi)$ for any $X\in {\mathfrak X}(M)$, which is equivalent to
\begin{equation*}
\big(|\nabla f|^2\xi(\rho)-|\xi|^2\nabla f(\lambda)\big)\big(|\xi|^2g(\nabla f, X)-\xi(f)g(\xi,X)\big)=0
\end{equation*}
for any $X\in {\mathfrak X}(M)$. Then we can state

\begin{proposition}
Let $\xi$ be a nontrivial closed (in particular, gradient) vector field on an $n$-dimensional connected gradient almost Ricci soliton $(M,\nabla f, \lambda,g)$ satisfying (\ref{2}) with $\rho$ nowhere zero. Then
\begin{equation}\label{8}
|\nabla f|^2\xi(\rho)=|\xi|^2\nabla f(\lambda) \ \ \ \textrm{or} \ \ \ |\xi|^2\nabla f=\xi(f)\xi.
\end{equation}
\end{proposition}

\noindent And we have
\begin{theorem}
Let $\xi$ be a nontrivial closed (in particular, gradient) affine conformal Killing vector field on an $n$-dimensional complete and connected gradient Ricci soliton $(M,\nabla f, \lambda,g)$ satisfying (\ref{2}).
If $\xi$ is not collinear with $\nabla f$, then $M$ is isometric to the Euclidean space.
\end{theorem}
\begin{proof}
From (\ref{8}), under the hypotheses, it follows that $\xi(\rho)=0$, hence, according to (\ref{c}), $\rho$ must be a constant.
Then, defining $h:=\frac{1}{2\rho}|\xi|^2$, we get $\Hess(h)=\rho g$, therefore, we obtain the conclusion by Tashiro's Theorem \ref{tasi}.
\end{proof}

The above theorem extends the result from \cite{bar}.

\bigskip

Further, we shall make some remarks in the gradient case.
Let $\mathcal L_{\nabla f}g=2\rho g+K$ with $\nabla K=0$ and $f$ and $\rho$ smooth functions on $(M,g)$.
Then
\[
\div(\mathcal L_{\nabla f}g)=2d\rho.
\]
But
\begin{equation}\label{11}
\div(\mathcal L_{\nabla f}g)=2d(\Delta(f))+2i_{Q(\nabla f)}g
\end{equation}
and we deduce
\begin{equation}\label{12}
\Ric(\nabla f,\nabla f)+\nabla f(\Delta(f))=g(\nabla f,\nabla \rho).
\end{equation}
Then we can state

\begin{proposition}
A gradient affine conformal Killing vector field $\nabla f$ on a compact and connected Riemannian manifold $(M,g)$ satisfying (\ref{2}) and
$\int_M\nabla f(\Delta(f))\geq 0$
is an affine Killing vector field.
\end{proposition}
\begin{proof}
By integrating the Bochner's formula \cite{ya} we infer
\[
-\int_M|\Hess(f)|^2=\int_M\big(\Ric(\nabla f,\nabla f)+\nabla f(\Delta(f))\big)=\int_Mg(\nabla f,\nabla \rho).
\]
Since a gradient vector field is closed, and since it satisfies (\ref{2}), we have $Q(\nabla f)=-(n-1)\nabla \rho$, so $$\Ric(\nabla f,\nabla f)=-(n-1)g(\nabla f, \nabla \rho)$$ and we get $\nabla f(\Delta(f))=ng(\nabla f, \nabla \rho)$ from (\ref{12}). Using the hypotheses, we deduce that $\Hess(f)=0$ and $d\rho=0$ (so, $\rho$ must be a constant).
\end{proof}

On the other hand, from (\ref{11}) we get
\begin{equation}\label{13}
\Ric(\nabla f,\nabla \rho)+\nabla \rho(\Delta(f))=|\nabla \rho|^2
\end{equation}
and we can state

\begin{proposition}
A gradient affine conformal Killing vector field $\nabla f$ on a connected Riemannian manifold $(M,g)$ satisfying (\ref{2}) and $\nabla \rho(\Delta(f))\leq 0$
is an affine Killing vector field.
\end{proposition}
\begin{proof}
Similarly as before, we have $Q(\nabla f)=-(n-1)\nabla \rho$, so $$\Ric(\nabla f,\nabla \rho)=-(n-1)|\nabla \rho|^2$$ and we get
$\nabla \rho(\Delta(f))=n|\nabla \rho|^2$ from (\ref{13}), hence we obtain the conclusion.
\end{proof}

\section{Some triviality results for almost Ricci solitons}

We shall further provide conditions for an affine conformal Killing vector field to be affine Killing, and the manifold to be an Einstein manifold.

For an almost Ricci soliton $(M,\xi,\lambda,g)$ with affine conformal Killing potential vector field satisfying (\ref{h}),
we have
\begin{equation}\label{k}
K=2(\lambda-\rho)g-2\Ric.
\end{equation}
By taking the divergence in (\ref{k}), we obtain
$dr=2d(\lambda-\rho)$, and by taking the trace in the same equation and differentiating then, we get $dr=nd(\lambda-\rho)$. If $n>2$ we find that $\lambda-\rho$ is a constant, hence $r$ must be a constant, too, and we can prove the following results.

\begin{proposition}\label{t1}
Let $(M,\xi,\lambda,g)$ be a connected $n$-dimensional Ricci soliton ($n>2$) with affine conformal Killing potential vector field satisfying (\ref{h}).
Then $\xi$ is an affine Killing vector field and the manifold is a trivial Ricci soliton.
\end{proposition}
\begin{proof}
Since $\lambda$ is a constant,
we get that $\rho$ must be a constant, too, and we deduce that
$\mathcal L_{\xi}\nabla=0$, hence $\mathcal L_{\xi}g=0$ and $\Ric=\lambda g$.
\end{proof}

\noindent Then we conclude

\begin{proposition}
There is no connected nontrivial Ricci soliton having affine conformal Killing potential vector field which is not affine Killing.
\end{proposition}

Concerning almost Ricci solitons, we have

\begin{proposition}
Let $(M,\xi,\lambda,g)$ be a compact and connected $n$-dimensional almost Ricci soliton ($n>2$) with affine conformal Killing potential vector field satisfying (\ref{h}).
If $r\neq 0$ and $$r(n\lambda-r)\leq 0,$$ then $\xi$ is an affine Killing vector field and the manifold is a trivial Ricci soliton.
\end{proposition}
\begin{proof}
From (\ref{k}) we get
\[
H=2(\lambda-\rho)I-2Q.
\]
Since $r$ is a constant and \cite{du}
\begin{equation}\label{b}
\mathcal L_{\xi}r=-2(n-1)\Delta(\rho)-2\rho r-\trace(HQ),
\end{equation}
we get
\[
|Q|^2-\frac{r^2}{n}=\frac{r(n\lambda-r)}{n}+(n-1)\Delta(\rho),
\]
which, by integration and using Schwartz's inequality, $|Q|^2-\frac{r^2}{n}\geq 0$, implies that $Q=\frac{r}{n}I$, hence $\Ric=\frac{r}{n}g$ and the manifold is Einstein. From $r(n\lambda-r)=0$ and $r\neq 0$, we get $r=n\lambda$ (so, $\lambda$ must be a constant and $\rho$ must be a constant, too). Then $\xi$ is an affine Killing, hence Killing, vector field.
\end{proof}

\begin{proposition}
Let $(M,\xi,\lambda,g)$ be a compact and connected $n$-dimensional gradient almost Ricci soliton ($n>2$) with affine conformal Killing potential vector field satisfying (\ref{h}).
If $$\Ric(\xi,\xi)\geq -n\xi(\lambda),$$ then $\xi$ is an affine Killing vector field and the manifold is a trivial Ricci soliton.
\end{proposition}
\begin{proof}
By tracing the soliton equation we obtain
\[
\div(\xi)=n\lambda-r
\]
and since $r$ is a constant, we further get
\[
\xi(\div(\xi))=n\xi(\lambda).
\]

Let $\xi=\nabla f$. Then by integrating the Bochner's formula \cite{ya} we infer
\[
-\int_M|\Hess(f)|^2=\int_M\big(\Ric(\xi,\xi)+n\xi(\lambda)\big)
\]
and using the hypotheses, we deduce that $\Hess(f)=0$ and $\Ric=\lambda g$ (since the scalar curvature $r=n\lambda$ is constant, then $\lambda$ must be constant and $\rho$ must be a constant, too) and we deduce that $\mathcal L_{\xi}\nabla=0$, hence $\mathcal L_{\xi}g=0$.
\end{proof}

Finally, we prove the following result

\begin{proposition}\label{14}
Let $(M,\xi,\lambda,g)$ be a compact and connected $n$-dimensional almost Ricci soliton ($n>2$) with closed (in particular, gradient) affine conformal Killing potential vector field satisfying (\ref{h}).
If
\[
\int_M\Ric(\xi,\xi)\geq n\int_M\rho\big((n-1)\rho+\trace(H)\big)+\frac{n-1}{4n}\int_M\big(\trace(H)\big)^2,
\]
then $\xi$ is a conformal Killing vector field and the manifold is Einstein.
\end{proposition}
\begin{proof}
From (\ref{k}) we infer
$H=2(\lambda-\rho)I-2Q$ and we get
\[
\trace(H)=2\big(n(\lambda-\rho)-r\big), \ \ |H|^2=4\Big(|Q|^2-\frac{r^2}{n}\Big)+\frac{\big(\trace(H)\big)^2}{n}.
\]
Also
\[
|\mathcal{L}_{\xi}g|^2=4n\rho^2+2\rho\trace(H)+|H|^2
\]
and computing $|\nabla \xi|^2$ (for $J=0$), we find
\[
|\nabla \xi|^2=n\rho^2+\rho\trace(H)+\frac{1}{4}|H|^2.
\]
Replacing now also $\div(\xi)$ from (\ref{15}) into Yano's integral formula \cite{ya}
\[
\int_M\Big(\Ric(\xi,\xi)+\frac{1}{2}|\mathcal{L}_{\xi}g|^2-|\nabla \xi|^2-\big(\div(\xi)\big)^2\Big)=0,
\]
we obtain
\[
\int_M\Big(\Ric(\xi,\xi)-n\rho\big((n-1)\rho+\trace(H)\big)-\frac{n-1}{4n}\big(\trace(H)\big)^2\Big)
=-\int_M\Big(|Q|^2-\frac{r^2}{n}\Big),
\]
which, by using the hypotheses and Schwartz's inequality, $|Q|^2-\frac{r^2}{n}\geq 0$, implies that $Q=\frac{r}{n}I$, hence $\Ric=\frac{r}{n}g$, the manifold is Einstein and
$\mathcal{L}_{\xi}g=2\big(\lambda-\frac{r}{n}\big)g$.
\end{proof}

\noindent In particular, if $\trace(H)=0$ we obtain
\[
\int_M\Big(\Ric(\xi,\xi)-n(n-1)\rho^2\Big)=0
\]
from (\ref{10}) and we can state

\begin{proposition}
Let $(M,\xi,\lambda,g)$ be a compact and connected $n$-dimensional almost Ricci soliton ($n>2$) with closed (in particular, gradient) affine conformal Killing potential vector field satisfying (\ref{h}).
If $\trace(H)=0$,
then $\xi$ is a conformal Killing vector field and the manifold is Einstein.
\end{proposition}

\bigskip

\textit{Adara M. Blaga}

\textit{Department of Mathematics}


\textit{West University of Timi\c{s}oara}


\textit{Timi\c{s}oara 300223, Rom\^{a}nia}

\textit{adarablaga@yahoo.com}

\bigskip

\textit{Bang-Yen Chen}

\textit{Department of Mathematics}

\textit{Michigan State University}

\textit{East Lansing, Michigan 48824-1027, USA}

\textit{chenb@msu.edu}

\end{document}